\documentclass[reqno]{amsart}

\usepackage{latexsym,amsfonts}
\usepackage{amssymb, amscd}
\usepackage{pgf}
\usepackage{tikz}
\usepackage{tikz-cd}
\usetikzlibrary{arrows, matrix}
\usepackage{fge}
\usepackage{amsmath,amsthm}          
\usepackage{euscript}                
\usepackage{epsfig}
\usepackage{appendix}
\usepackage{mathrsfs}
\usepackage{calrsfs}

\usepackage[bookmarks,
bookmarksnumbered,%
 colorlinks=true,%
 linkcolor=blue,%
 citecolor=red,%
 filecolor=blue,%
 urlcolor=blue,%
]
{hyperref}

\theoremstyle{plain} \numberwithin{equation}{section}
\newtheorem{main}{Theorem}

\newtheorem{thm}{Theorem}[section]
\newtheorem{cor}[thm]{Corollary}
\newtheorem{prop}[thm]{Proposition}

\newtheorem{lemma}[thm]{Lemma}

\newtheorem{coro}[main]{Corollary}
\newtheorem{rmk}[thm]{Remark}
\newtheorem{ex}[thm]{Example}
\newtheorem{defn}[thm]{Definition}

\newcommand{\bi}{\begin{itemize}}
\newcommand{\ei}{\end{itemize}}
\newcommand{\bp}{\begin{proof}}
\newcommand{\ep}{\end{proof}}

\DeclareMathOperator{\spec}{Spec}
\DeclareMathOperator{\proj}{Proj}
\DeclareMathOperator{\lcm}{lcm}

\def\ra{\rightarrow}
\def\cal{\mathcal}

\def\CC{\mathbb{C}}
\def\PP{\mathbb{P}}
\def\QQ{\mathbb{Q}}
\def\ZZ{\mathbb{Z}}

\def\DD{\cal D}

\def\OO{\cal O} 
\def\XX{\cal X}

\def\VV{\cal V}
\def\SS{\cal S}
\def\WW{\cal W}
\def\TT{\cal T}

\def\s-{\setminus}

\begin{document}

\title{On the classification of ALE  K\"ahler manifolds}

\author[Hein]{Hans-Joachim Hein}
\address{        
       Department of Mathematics, Fordham University, Bronx, NY 10458, USA}
\email{hhein@fordham.edu}

\author[R\u asdeaconu]{Rare\c s R\u asdeaconu}
\address{        
        Department of Mathematics, 1326 Stevenson Center, Vanderbilt University, Nashville, TN, 37240, USA}    
        \email{rares.rasdeaconu@vanderbilt.edu}
            
\author[\c Suvaina]{Ioana \c Suvaina}
\address{        
        Department of Mathematics, 1326 Stevenson Center, Vanderbilt University, Nashville, TN, 37240, USA}
\email{ioana.suvaina@vanderbilt.edu}

\date{\today}



\begin{abstract}
The underlying complex structure of an ALE  K\"ahler manifold is exhibited as a 
resolution of a deformation of an isolated quotient singularity. As a consequence, there 
exist only finitely many diffeomorphism types of minimal ALE K\"ahler surfaces with a given group at infinity.
\end{abstract}

\maketitle

\thispagestyle{empty}

\section{Introduction}

An asymptotically locally Euclidean (ALE) K\"ahler  manifold of dimension $n$ is a complete
complex K\"ahler manifold which, outside of a compact set, is modeled by a quotient of the Euclidean space
by a finite subgroup $G$ of $U(n).$
The ALE scalar-flat K\"ahler (sfK) surfaces play a 
central r\^ole in the study of the Gromov-Hausdorff limit of constant scalar curvature 
K\"ahler surfaces. However, their classification is missing and, as Chen-LeBrun-Weber 
pointed out \cite[page 1164]{clw}, this ``seems a daunting problem".  

Simply connected ALE Ricci-flat  K\"ahler surfaces were constructed by Gibbons-Hawking 
\cite{g-h}, Hitchin \cite{hitchin}, and  
Kronheimer \cite{kron2},  while non-simply connected examples are provided by free quotients of 
some of Hitchin's \cite{hitchin} (see also \cite{ALE}). 
 Kronheimer \cite{kron3}, 
\c Suvaina \cite{ALE}  and Wright \cite{wright} showed that these constructions yield the complete list of  ALE 
Ricci-flat K\"ahler surfaces. 
In higher dimensions, 
ALE Ricci-flat K\"ahler metrics were constructed by Calabi \cite{calabi1}, Joyce \cite{joyce2} 
and Tian-Yau \cite{ty2}. There are 
currently many known non-Ricci-flat ALE sfK metrics. In complex dimension $2$, such 
metrics were constructed by LeBrun  \cite{lebrun-mass,  lebrun-ansatz}, Joyce \cite{joyce}, 
Calderbank-Singer \cite{ca-si},
Honda \cite{honda1, honda2},  Lock-Viaclovsky \cite{lock-viaclovsky} and Han-Viaclovsky \cite{h-v-def}. 
Very recently, in a paper \cite{h-v-existence} which appeared after this article was posted, 
Han and Viaclovsky showed that the minimal smoothings of a cyclic quotient singularity
admit ALE sfK metrics. Non-Ricci-flat ALE sfK metrics in higher dimensions 
were found by  Simanca \cite{simanca} 
and  Apostolov-Rollin \cite{a-r}. A common feature of all of the 
known examples of ALE sfK manifolds is that  the underlying complex structure is a 
resolution of a deformation of a quotient singularity. 
We address here the classification of ALE sfK  manifolds by placing it in the broader context of 
ALE K\"ahler manifolds and show that this common feature is in fact a general phenomenon.

\begin{main}
\label{DiffTypeALE}
Any ALE K\"ahler manifold asymptotic to $\CC^n/G$ is biholomorphic to a resolution 
of a deformation of the isolated quotient singularity $(\CC^n/G,0).$
\end{main}

\begin{rmk}
In his book, Joyce \cite[Section 8.9]{joyce-book} claims that such a result would follow from the 
results previously discussed in Section 8 of {\it{op.cit.}} However, a proof is not given.
\end{rmk}

In dimension at least $3$, by Schlessinger's theorem  \cite{schless},
isolated quotient singularities are rigid under deformation.
Hence, Theorem \ref{DiffTypeALE} immediately implies:

\begin{coro}
\label{ALEsfK}
If $n\geq 3,$ every ALE K\"ahler manifold asymptotic to $\CC^n/G$ is biholomorphic 
to a resolution of the isolated singularity $\CC^n/G.$
\end{coro}

\begin{rmk}
This recovers 
and refines the rigidity results of Hein-LeBrun in  \cite[Section 2]{hein-lebrun} via a 
different method. In this direction, see also the very closely related results of 
Hamilton \cite[page 1--2]{hamilton} and Tian \cite[Section 4]{tian}.
\end{rmk}

An important idea in algebraic geometry, which goes back to Zariski \cite{zar}, 
is that the finite generation of rings of sections can be used to derive information about birational 
models of algebraic varieties. The proof of Theorem \ref{DiffTypeALE} is yet another illustration of this idea.
Any ALE K\"ahler manifold admits a natural divisorial analytic compactification due to Li \cite{chili} and Hein 
and LeBrun \cite{hein-lebrun}. The divisor at infinity is isomorphic to 
$\PP^{n-1}/G,$ 
and the ring of sections of its normal bundle is isomorphic to the ring of invariants  
of the isolated quotient singularity $\CC^n/G$. This suffices to show that the starting 
ALE K\"ahler manifold is in fact a resolution of an affine algebraic variety. 
To identify this affine algebraic variety we appeal to the ``sweeping out the 
cone with hyperplane sections" technique of Pinkham \cite{pink}.

The deformation theory of isolated quotient singularities in complex dimension $2$
is well-understood and it allows us to prove the following result:

\begin{main}
\label{finiteness}
For every finite subgroup $G\subset U(2)$ containing no reflections, 
there exist only finitely many diffeomorphism types underlying minimal  ALE 
K\"ahler surfaces which are asymptotic to $\CC^2/G.$ 
\end{main}

By a minimal surface we mean a smooth complex surface 
 which does not contain any holomorphic sphere of self-intersection 
$-1$.
In contrast to dimension $2$, in higher dimensions it is a difficult task 
to impose a minimality condition on a resolution.  For this reason 
we do not attempt here to explore a generalization of Theorem 
\ref{finiteness} to higher dimensions. 

The construction used in the  proof of Theorem \ref{DiffTypeALE} also allows us to 
give a new proof of the characterization of asymptotically Euclidean K\"ahler 
manifolds, i.e., of ALE K\"ahler manifolds for which the group at infinity is trivial, 
a result originally due to Hein and LeBrun 
\cite[Proposition 2.12 and Proposition 4.3]{hein-lebrun}.

\begin{coro}
\label{AEK}
Any asymptotically Euclidean  K\"ahler manifold $M$ of complex dimension 
$n$ admits a proper holomorphic map onto $\CC^n$ that is an isomorphism 
away from compact subsets of $M$ and $\CC^n$.
\end{coro}

\subsection*{Acknowledgements} The first author is partially supported by 
NSF grant DMS-1745517. The second author acknowledges the support 
of the Simons Foundation's "Collaboration Grant for Mathematicians", 
Award \#281266, while the the third author is partially supported by the 
NSF grant DMS-1710970. The authors are thankful to Lucian B\u adescu, 
Claude LeBrun and Jeff Viaclovsky for their comments on the paper, and to J\'anos Koll\'ar 
for his help with the proof of Theorem \ref{finiteness}. 
The third author would like to thank IHP for hospitality where this project started.

\section{Divisorial analytic compactifications of ALE K\"ahler manifolds}
\label{1}
 
There exist several definitions of an ALE K\"ahler manifold in the literature,
each of them imposing various fall-off conditions on the metric at infinity. 
We are working in the following setting:

\begin{defn}
\label{ALE}
A complete K\"ahler manifold $(M, g, J)$ of complex dimension $n$ is said to be 
asymptotically locally Euclidean (ALE) K\"ahler if there exists a compact 
subset $K\subset M$ such that $M\setminus K$ has finitely many 
connected components $(M\setminus K)_i$, $i\in I,$ and for each $i \in I$ there exists a finite subgroup $G_i\subset U(n)$ acting freely on $\CC^n\setminus B_R(0),$ $B_R(0)=\{x\in\CC^n: |x|\leq R\}$,
and a {\it diffeomorphism} 
$$
\pi_i:(M\setminus K)_i\to \left(\CC^n\setminus B_R(0)\right)/G_i,
$$ 
such that for some $\tau>0$,
\begin{equation*}
\label{fall-off-conditions}
(\pi_i)_*g-g_{\CC^n}\in C^{2,\alpha}_{-\tau}~\text{and}~(\pi_i)_*J-J_{\CC^n}\in C^{2,\alpha}_{-\tau}.
\end{equation*}
\end{defn}

Here, $g_{\CC^n}$ and $J_{\CC^n}$ denote the standard flat K\"ahler structure on $\CC^n$. Moreover, the weighted H\"older space $C^{2,\alpha}_{-\tau}$
consists of $C^{2,\alpha}$ functions $f$ such that
$$
\left( \sum_{j=0}^2 |x|^j |\nabla^j f(x)|\right)
+|x|^{2+\alpha}[\nabla^2 f ]_{C^{0,\alpha}( B_{|x|/10}(x))}
= O(|x|^{-\tau}).
$$
This definition can naturally be extended to tensor fields.

The connected components of $M\setminus K$ are called ends. 
Hein and LeBrun prove in \cite[Proposition 2.5 and Proposition 4.2]{hein-lebrun} 
that an ALE K\"ahler manifold has in fact only one end. (This is done under stronger 
fall-off assumptions at infinity than we have imposed here, but the alternative proofs 
mentioned in \cite[page 195, Remark]{hein-lebrun} actually only require the weaker 
fall-off rate $C^1_{-\tau}$, $\tau > 0$, for $g$ and $J$.)  For the rest of the 
presentation, we will tacitly assume that $M$ has only one end.  

Our definition for ALE K\"ahler manifolds is inspired by the work on compactifications of 
asymptotically conical manifolds of Li \cite{chili} (see also \cite{hhn} and \cite{conlon-hein-3}).
The stronger fall-off conditions at infinity in  \cite{hein-lebrun}, where $\tau$ is required to be larger than 
$n-1$, are necessary for the definition of the ADM mass, but can be removed from the construction 
of the compactification. 

\subsection{A compactification of ALE K\"ahler manifolds}
\label{compact-ALE}

Let $(M,g,J)$ be an ALE K\"ahler manifold of complex dimension $n\geq 2$ in the sense of Definition \ref{ALE}. Let $M_{\infty}$ 
be the end of $M,$ and let $\widetilde M_{\infty}$ be the universal cover of $M_{\infty}.$ Then, after a $C^\infty$ identification of $\widetilde M_{\infty}$ with $\CC^n\setminus B_R(0)$, the tensors $g,J$ converge to their Euclidean counterparts at rate  $C^{2,\alpha}_{-\tau}$.  Here, the parameter $\tau>0$ is arbitrary.

For $\tau>n-1,$ Hein and LeBrun show \cite[Lemma 2.1 and Lemma 4.1]{hein-lebrun} that there 
exists a complex manifold $\cal X$ of dimension $n$ containing an embedded complex 
hypersurface $\Sigma\simeq \CC\PP^{n-1},$ with normal bundle of degree $+1,$ such that $\widetilde M_{\infty}$ is biholomorphic to $\cal X\setminus \Sigma.$ 
The same statement can be proved more generally for arbitrary values of the parameter $\tau>0$ by using the method of Haskins-Hein-Nordstr\"om \cite{hhn}, 
which was already used in \cite{hein-lebrun} to treat the case $n=2$ and $1<\tau\leq 3/2.$ 
The method of Li \cite{chili} yields the same result assuming $C^{2n}_{-\tau}$ instead of 
$C^{2,\alpha}_{-\tau}$ fall-off of $g$ and $J$, again for an arbitrary $\tau>0$. 
In fact, it is natural to expect that $C^{1,\alpha}_{-\tau}$ for any $\tau > 0$ suffices in 
all dimensions, i.e., that the $C^{2,\alpha}_{-\tau}$ condition of Definition \ref{ALE} can 
be relaxed to $C^{1,\alpha}_{-\tau}$ for our purposes. However, the methods available in the literature 
do not seem to be sufficient for such a strengthening.

Since $H^1(\CC\PP^{n-1},\OO_{\CC\PP^{n-1}}(1)) = 0,$ standard 
deformation theory arguments \cite{kodaira} imply that there is a 
complete analytic family of compact complex submanifolds of dimension 
$h^0(\CC\PP^{n-1},\OO_{\CC\PP^{n-1}}(1)) = n,$ which represents all 
small deformations of $\Sigma \subset \XX$ through compact complex 
submanifolds. Since $\CC\PP^{n-1}$ is rigid under deformations and 
$h^{0,1}(\CC\PP^{n-1}) = 0,$ we may assume, by shrinking the size 
of the family if necessary, that every submanifold in the family is 
biholomorphic to $\CC\PP^{n-1},$ and has normal bundle 
$\OO_{\CC\PP^{n-1}}(1).$ The union of these submanifolds fills out 
some open neighborhood $\cal U$ of $\Sigma$ in $\XX.$ If $n\geq 3,$ such a  
neighborhood  is in fact biholomorphic to a neighborhood of the hyperplane
$\CC\PP^{n-1}\subset \CC\PP^{n}$ 
\cite[Lemma 2.2]{hein-lebrun}. In dimension $2$, only a weaker result 
holds. Namely, the first infinitesimal neighborhood of $\Sigma\subset \XX$ 
is isomorphic to the first infinitesimal neighborhood of 
$\CC\PP^{1}\subset \CC\PP^2,$ i.e., the normal bundle of 
$\Sigma$ in $\XX$ 
is isomorphic to the normal bundle of a line in 
$\CC\PP^2.$ Hein and LeBrun prove a stronger result 
\cite[Lemma 4.4]{hein-lebrun}, but this isomorphism between the first 
infinitesimal neighborhoods suffices for our purpose in all dimensions. 

The action of $G$ extends continuously to $\XX$ and since it is holomorphic 
on the complement of $\Sigma,$ the induced action is actually holomorphic,
and it sends $\Sigma$ to itself. Moreover, we may assume that the open set 
$\cal U\subset \XX$ is $G$-invariant. Since the automorphism group of the 
first infinitesimal neighborhood of $\Sigma$ in $\XX$ is isomorphic to  the 
automorphism group of the first infinitesimal neighborhood of 
$\CC\PP^{n-1}\subset \CC\PP^n,$ which is $GL(n;\CC),$ the action of the 
finite group $G$ on $T\XX|_{\Sigma}$ is realized by that of a finite subgroup of 
the maximal compact subgroup $U(n)\subset GL(n;\CC).$

Let $G_0\subset G$ 
be the normal subgroup corresponding to the center $U(1)\subset U(n).$ This is exactly the subgroup of $G$ whose action on $\cal U$ fixes every point of $\Sigma$. The 
quotient $\overline{\cal U}=\cal U/G_0$ may then be viewed as a smooth manifold in such a way that the quotient map 
$r:\cal U\ra \overline{\cal U}$ is a branched cyclic cover ramified along $\Sigma.$ By construction, the fixed locus of the induced action of $G/G_0$ on $\overline{\cal U}$ is a proper complex subvariety of $\overline{\Sigma} = r(\Sigma) 
\subset \overline{\cal U}$, i.e., it has codimension at least $2$ in $\overline{\cal U}$. Moreover, the isotropy groups of the action of $G/G_0$ on $\overline{\cal U}$ are cyclic; indeed, they must inject into the structure group, $U(1)$, of the complex line bundle $N_{\overline{\Sigma}|\overline{\cal U}}$ because otherwise the fixed locus of $G/G_0$ in $\overline{\mathcal{U}}$ would contain $2$-disks transverse to $\overline{\Sigma}$.

Let 
$\widehat{\cal U}={\overline{\cal U}}/(G/G_0)$,
let $q: {\overline{\cal U}}\ra \widehat{\cal U}$ be
the quotient map, and let $D\subset \widehat{\cal U}$ be the image 
of $\overline{\Sigma} \subset \overline{\cal U}.$ The composition map  $\pi: \cal U\ra \widehat{\cal U}$ 
given by $\pi=q\circ r$ is an orbifold chart 
around the divisor $D.$ We can now compactify $M$ as a complex space $X,$ by 
gluing $\widehat{\cal U}$ to $M.$ The boundary
$X\setminus M$ is the divisor 
$D$, which is isomorphic to $\CC\PP^{n-1}/(G/G_0)$ as a complex orbifold. 
Furthermore, the space $X$ has at most quotient singularities in 
codimension 2, and its singular locus is contained in $D.$

In general, as $X$ may have singularities along $D,$ the sheaf $\OO_X(D)$ is not locally free, 
but is merely an orbi-line bundle \cite{baily}, and the same is true for $N_{D|X}\simeq \OO_D(D).$ 
However, this is not a serious obstacle, as will be seen later. For the moment, notice that since 
$N_{\Sigma|\XX}\simeq N_{\CC\PP^{n-1}|\CC\PP^n}\simeq \OO_{\CC\PP^{n-1}}(1),$ 
the  orbi-line bundle $N_{D|X}\simeq \OO_D(D)$ is ample, i.e., it admits 
a hermitian metric of positive curvature. When 
$n\geq 3,$ this implies that  $X$ is a K\"ahler orbifold \cite[Lemma 2.4]{hein-lebrun}. If 
$n=2,$ since $D^2>0$, it follows from \cite[Chapter IV, Theorem 6.2]{bhpv} that $X$ is projective. 

The action of $G$ on $T\XX|_{\Sigma}$ induces 
an action on the section ring of the normal bundle,
$R(\Sigma,\OO_{\Sigma}(\Sigma))\simeq R(\CC\PP^{n-1},\OO_{\CC\PP^{n-1}}(1)),$  
where 
$$
R(\Sigma,\OO_{\Sigma}(\Sigma))=\bigoplus_{k\geq 0} 
H^0(\Sigma, N_{\Sigma|\XX}^{\otimes k})
$$ 
and 
$$
R(\CC\PP^{n-1},\OO_{\CC\PP^{n-1}}(1))=\bigoplus_{k\geq 0} 
H^0(\CC\PP^{n-1},\OO_{\CC\PP^{n-1}}(k)).
$$
Therefore, there exists an isomorphism of invariant subrings 
$$
R(\Sigma,\OO_{\Sigma}(\Sigma))^G\simeq R(\CC\PP^{n-1},
\OO_{\CC\PP^{n-1}}(1))^G.
$$
We define the section ring of $\OO_D(D)$ as  
$$
R(D, \OO_{D}(D)):=\bigoplus_{k\geq 0} H^0(D,\OO_{D}(kD)).
$$
Then 
$
R(D,  \OO_{D}(D))\simeq R(\Sigma,\OO_{\Sigma}(\Sigma))^G,
$
and since $R(\CC\PP^{n-1},\OO_{\CC\PP^{n-1}}(1))$ is the polynomial ring 
$\CC[z_1,\dots,z_n]$ in $n$ variables, we have an isomorphism
\begin{equation}
\label{normal-ring}
R(D,  \OO_{D}(D))\simeq 
\CC[z_1,\dots,z_n]^G,
\end{equation}
the latter being the coordinate ring of the isolated quotient 
singularity $\CC^n/G.$

The following theorem summarizes  the  discussion above: 
\begin{thm}
\label{summary}
Let $M$ be an ALE K\"ahler manifold of complex dimension $n$, asymptotic at 
infinity to $\CC^n/G,$ where $G\subset U(n)$. Then $M$ has only one end and 
there exists an $n$-dimensional compact K\"ahler orbifold $X$ and a connected suborbifold divisor
$D\subset X$ with the following properties:
\begin{itemize}
\item[i)] $M=X\setminus D.$ 
\item[ii)] $D$ is isomorphic to $\CC\PP^{n-1}/(G/G_0)$ as a complex orbifold.
\item[iii)] The normal orbi-bundle $N_{D|X}\simeq \OO_D(D)$  is ample and its 
ring of sections $R(D, \OO_D(D))$ is isomorphic to the coordinate ring of the 
isolated quotient singularity $\CC^n/G.$
\end{itemize}
\end{thm} 

\subsection{Algebro-geometric properties of the compactification}

The compactification constructed above exhibits several features which 
will play a key role in the proof of Theorem \ref{DiffTypeALE}. 
Closely related results were proved by Conlon-Hein in \cite[Section 2.3]{conlon-hein-2} and 
\cite[Section 2]{conlon-hein-3} using different arguments.

For the convenience of the reader, we start by recalling some terminology.

\begin{defn} Let $S$ be a complex variety. Let $E$ be a Weil divisor on $S$.
\begin{itemize}
\item[i)] We say that 
$E$ is $\QQ$-Cartier if there exists an integer $m>0$ such that $mE$ is Cartier, i.e., 
such that the sheaf $\OO_S(mE)$ is a line bundle. 
\item[ii)] We say that $S$ is $\QQ$-factorial if every Weil divisor on $S$ is $\QQ$-Cartier.
\end{itemize}
\end{defn}

\begin{defn}
We say that a normal variety $S$ has rational singularities if  $R^if_*\OO_{S'}=0$ 
for every resolution $f:S'\ra S$ of $S$, and for every $i>0.$
\end{defn}

\begin{lemma}\emph{(}\cite[Proposition 5.15]{km}\emph{)}
Let $S$ be a complex variety with quotient singularities only. Then $S$ is 
$\QQ$-factorial, and its singularities are rational.   
\end{lemma}

\begin{defn} Let $S$ be a complex variety.
\begin{itemize}
\item[i)] We say that a Cartier divisor $L$ on $S$  is big if there exist a positive constant
$c=c(L)$ and a  positive integer $m_0$ such that 
$$
h^0(S, L^{\otimes m})\geq cm^{\dim_\CC S},
$$ 
for every $m\geq m_0.$ We say  
that a $\QQ$-Cartier divisor $E$ on $S$ is big if there exists a positive integer 
$k$ such that $kE$ is a big Cartier divisor.
\item[ii)] We say that a Cartier divisor $L$ is pseudo-ample 
if there exists  a positive integer $m_0$ such that $L^{\otimes m}$ is globally 
generated for $m\geq m_0.$ We say  
that a $\QQ$-Cartier divisor $E$ on $S$ is pseudo-ample  if there exists a 
positive integer $k$ such that $kE$ is a pseudo-ample Cartier divisor.
\end{itemize}
\end{defn}

\begin{defn}
A complex variety $S$ is called Moishezon if the transcendence degree of 
its field of meromorphic function is $\dim_\CC S.$
\end{defn}

The following criterion is a standard result in the smooth setting. For 
convenience of the reader, we include a proof in the orbifold setting.

\begin{lemma}
\label{orbi-M-criterion}
An orbifold $S$ is Moishezon if and only if it admits a big $\QQ$-Cartier divisor.
\end{lemma}

\begin{proof} Let $\pi: T\ra S$ be a resolution of singularities. Since the 
transcendence degree of the function field is a bimeromorphism invariant, we see that $S$ is Moishezon 
if and only if $T$ is Moishezon. Moreover, from the definition, we see that $S$ 
admits a big $\QQ$-Cartier divisor if and only if it admits a big Cartier divisor.

Let $L_S$ be a big Cartier divisor on $S.$ Since $T$ is normal and 
$\pi$ has connected fibers, by the projection formula 
$$
h^0(T, \pi^*L_S^m)=h^0(S, \pi_*\pi^*L_S^m)=h^0(S, L_S^m)\geq cm^{\dim_\CC S},
$$
for some $c>0$ and $m$ sufficiently large. Therefore $\pi^*L_S$ 
is a big Cartier divisor on $T.$ Hence  $T$ is Moishezon 
\cite[Theorem 2.2.15]{ma-ma}, and so $S$ is Moishezon as well.

Conversely, since $S$ is Moishezon then $T$ is Moishezon as well. As above, by 
\cite[Theorem 2.2.15]{ma-ma}, there exists a big line bundle $L_T$ on $T$. 
Replacing $L_T$ by a large tensor power, we may assume $h^0(T,L_T)>0,$ 
and so there exists an effective divisor $H$ such that $L_T=\OO_T(H).$ Let $F=\pi(H),$ 
and discard all components of codimension strictly larger than one. Since $L_T$ is big, 
the divisor $H$ is not $\pi$-exceptional and so $F$ is a non-trivial Weil divisor on $S.$  
Moreover, since $S$ has only quotient singularities, there exists $k>0$ such that $kF$ 
is Cartier \cite[Proposition 5.15]{km}. Since $\pi$ is birational, for every positive 
integer $m$ there exists an effective $\pi$-exceptional divisor $E_m$ such that 
$\pi^*\OO_S(mkF)=\OO_T(mkH+E_m).$ This yields an injection of sheaves 
$\mathcal{O}_T(mkH) \hookrightarrow \pi^*\OO_S(mkF)$ by tensoring with a defining section of $\mathcal{O}_T(E_m)$.
By passing to global sections,
$$
h^0(S,\OO_S(mkF))=h^0(T,\pi^*\OO_S(mkF))\geq h^0(T, L_T^{mk})\geq c'(mk)^{\dim_\CC S},
$$
for some $c'>0$ and all $m\gg 0.$ Hence
$F$ is a big $\QQ$-Cartier divisor
on $S.$
\end{proof}

\begin{lemma} 
For every $m\geq 0,$ the following sequence is exact:
\begin{equation}
\label{b-ses}
0\ra\OO_X((m-1)D)\ra \OO_X(mD)\ra \OO_D(mD)\ra 0.
\end{equation}
\end{lemma}
\begin{proof}

In general, we have the exact sequence of the Weil divisor $D$: 
\begin{equation}
\label{ses-m=0}
0\ra\OO_X(-D)\ra \OO_X\ra \OO_D\ra 0,
\end{equation}
which yields (\ref{b-ses}) in the case $m=0.$ If $D$ is Cartier, then $\OO_X(mD)$ is locally free
(for example, if $X$ is smooth, i.e., when $G=G_0$). Since tensoring  by a locally free sheaf preserves 
exactness, we would then immediately see from (\ref{ses-m=0})  that (\ref{b-ses}) is true. However, our
$D$ is merely $\QQ$-Cartier, and so a different argument is required.

It is enough to show that (\ref{b-ses}) holds locally near $D.$ 
Let $\pi: \cal U\ra \widehat{\cal U}$ be the orbifold chart around $D,$ and let $\Sigma\subset \cal U$ be the 
hyperplane at infinity. Since $\cal U$  is smooth, for every $m\geq 0$ we have an exact sequence:
$$
0\ra \OO_{\cal U}((m-1)\Sigma)\ra\OO_{\cal U}(m\Sigma)\ra\OO_\Sigma(m\Sigma)\ra 0.
$$
Pushing forward onto $\widehat{\cal U}$ we get 
\begin{equation*}
0\ra \pi_*\OO_{\cal U}((m-1)\Sigma)\ra\pi_*\OO_{\cal U}(m\Sigma)\ra
\pi_*\OO_\Sigma(m\Sigma)\ra {R}^1 \pi_*\OO_{\cal U}((m-1)\Sigma).
\end{equation*}
Moreover, since $\pi$ is finite, ${R}^1 \pi_*\OO_{\cal U}((m-1)\Sigma)=0.$ Restriction to 
$G$-invariant subsheaves induces the exact sequence
\begin{equation}
\label{G-inv-ses}
0\ra \left(\pi_*\OO_{\cal U}((m-1)\Sigma)\right)^G\ra\left(\pi_*\OO_{\cal U}(m\Sigma))\right)^G
\ra\left(\pi_*\OO_\Sigma(m\Sigma)\right)^G\ra0.
\end{equation}
Since $\Sigma$ is $G$-invariant, and $D=\pi(\Sigma),$ we have
$(\pi_*\OO_{\cal U})^G\simeq \OO_{\widehat{\cal U}}$ and 
$(\pi_*\OO_\Sigma)^G\simeq \OO_D.$ Moreover,  by  
\cite[Chapter II, Section 7, Proposition 2]{mumford}, 
we have $(\pi_*\OO_{\cal U}(k\Sigma))^G$ $\simeq$ $\OO_{\widehat{\cal U}}(kD)$  
and $(\pi_*\OO_\Sigma(k\Sigma))^G\simeq  \OO_{D}(kD)$ for every $k\geq 0.$ 
Hence the exactness of the sequence (\ref{G-inv-ses}) implies 
the exactness of the sequence (\ref{b-ses}).
\end{proof}

\begin{lemma}
\label{XMoishezon}
Let $(X, D)$ be the compactification constructed in Theorem \ref{summary}. 
Then $X$ is Moishezon.
\end{lemma}
\begin{proof} Since $\Sigma \simeq \PP^{n-1},$ we have $H^i(\Sigma, \OO_\Sigma(m\Sigma))=
H^i(\PP^{n-1}, \OO_{\PP^{n-1}}(m))=0$ 
for any $i>0$ and $m\geq 0.$ As the restriction $\pi|_{\Sigma}:\Sigma\ra D$ is finite, by Leray's spectral sequence 
we also have $H^i(D,\OO_D{(mD)})=0$ for every $i>0$ and $m\geq 0.$ 

Passing to cohomology, from the short exact sequence of sheaves
$$
0\ra\OO_X{((m-1)D)}\ra\OO_X(mD)\ra  \OO_D{(mD)}\ra 0
$$
we see that $h^i(X,\OO_X{(mD)})=h^i(X,\OO_X((m-1)D))$ for $m>0$ and $i>1,$ 
while  $h^1(X,\OO_X{(mD)})\leq h^1(X,\OO_X{((m-1)D)})$ for $m>0.$
Therefore, the functions 
$$
m\mapsto h^i(X,\OO_X{(mD)}),~i\geq1,
$$ are bounded. Since 
$D^n>0,$ it follows from Riemann-Roch that 
$$
h^0(X,\OO_X(mD))\sim cm^{n},
$$ 
for some positive constant $c$ and for all $m \gg 0$. That means the $\QQ$-Cartier divisor $\OO_X(D)$ is big, 
and so, by Lemma \ref{orbi-M-criterion}, $X$ is Moishezon.
\end{proof}

\begin{prop}
\label{prop-vanishing}
Let $(X, D)$ be the compactification of an ALE K\"ahler manifold 
constructed in Theorem \ref{summary}. We have:
\begin{itemize}
\item[i)] $H^i(X,\OO_X)=0,~i>0.$
\item[ii)]  $H^i(X,\OO_X(mD))=0,\,m> 0,\,i>0.$
\item[iii)] The $\QQ$-Cartier divisor $\OO_X(D)$ is pseudo-ample. 
\item[iv)] $X$ is projective.
\end{itemize} 
\end{prop}
\begin{proof}
i)  The complex variety $X$ is a K\"ahler orbifold; cf.~Theorem \ref{summary}. Therefore, the orbifold de Rham 
cohomology of $X$ satisfies Hodge decomposition as well as
Hodge symmetry, $H^{p,q}(X)\simeq \overline{H^{q,p}(X)}.$ To prove the claim it is 
therefore sufficient to show that $H^0(X,\Omega_X^i)=0$ for 
every $i\geq 1,$ where $\Omega_X^i$ is the sheaf of holomorphic 
orbifold $i$-forms. We will do this using an argument similar to \cite[page 555]{lebrun-maskit}.

Let $s\in H^0(X,\Omega_X^i).$ The restriction of $s$ to $\widehat{\cal U}$ 
induces a holomorphic section $\pi^*s\in H^i(\cal U,\Omega^i_{\cal U}),$ 
where $\pi: \cal U\ra \widehat{\cal U}$ is the 
orbifold chart near $D.$   Recall now that  the hyperplane at infinity 
$\Sigma \simeq \CC\PP^{n-1}\subset \cal U$ has 
normal bundle $\OO_{\CC\PP^{n-1}}(1).$ This implies that any line 
$\ell\in \Sigma $ has normal bundle 
$N_{{\ell}|\cal U}\simeq \OO_{\CC\PP^1}(1)^{\oplus{(n-1)}}.$ 
Therefore, standard deformation theory shows that by deforming $\ell,$ 
one can sweep out a neighborhood of $\Sigma $ in $\cal U.$ This means that, by 
eventually shrinking $\cal U,$ for every point $p\in \cal U$ there exists a 
curve $C\simeq \CC\PP^1$ with normal bundle 
$\OO_{\CC{\PP^1}}(1)^{\oplus{(n-1)}}$ passing through $p.$ 
Since the restriction of any holomorphic $i$-form on $\cal U$ 
to a smooth rational curve with normal bundle 
$\OO_{\CC\PP^1}(1)^{\oplus{(n-1)}}$ is a section of the exterior product 
$$
\bigwedge\nolimits ^{i}\left(\OO_{\CC\PP^1}(-2)\oplus 
\OO_{\CC\PP^1}(-1)^{\oplus{(n-1)}}\right),
$$
such a restriction must vanish identically. As such curves sweep out the 
entire $\cal U,$ it follows that $\pi^*s$ vanishes identically on $\cal U.$ 
Hence, $s=0$ on $\widehat{\cal U}$, which implies 
$s=0$ on $X$ by analytic continuation.

ii) As in the proof of Lemma \ref{XMoishezon} we can infer 
that 
$$
h^i(X,\OO_X(mD))=h^i(X,\OO_X((m-1)D))~{\text{for every}}~m>0,\, i>1,
$$
and
$$
h^1(X,\OO_X(mD))\leq h^1(X,\OO_X((m-1)D))~{\text{for every}}~m>0.
$$ 
By induction, and using the fact that 
$H^i(X,\OO_X)=0$ for $i>0$ from item i) above, it now follows that 
$H^i(X,\OO_X(mD))=0$ for all $i>0$ and $m > 0$ as well.

iii) Let $k$ be a positive integer such that $\OO_X(kD)$ is Cartier. 
Since in this case $\OO_D(kD)$ is also Cartier and  ample, 
$\OO_D(mkD)$ is globally generated for $m$ large enough. Thus, for every point
$x \in \text{supp}(D)$, there exists a global section $s \in H^0(D, \mathcal{O}_D(mkD))$ 
such that $s(x) \neq 0$. From item ii) above, we 
know that there is an exact sequence of vector spaces 
\begin{equation*}
H^0(X,\OO_X(mkD))
\ra H^0(D,\OO_D(mkD))\ra 0. 
\end{equation*}
Thus, $s$ extends to a global section of $\mathcal{O}_X(mkD)$. 
Using also a defining 
section of $mkD$ to cover the case that $x \not\in \text{supp}(D)$, we see that for all $x \in X$ there exists a section 
$s \in H^0(X,\OO_X(mkD))$ such that $s(x) \neq 0$.  Thus, $\OO_X(mkD)$ 
is  globally generated for $m\gg 0.$ 

iv) If $n=2,$ this follows from \cite[Chapter IV, Theorem 6.2]{bhpv}. If $n\geq 3,$ we proved in Lemma 
\ref{XMoishezon} that $X$ is Moishezon, while by 
Theorem \ref{summary} we also know that it is K\"ahler. 
Since $X$ has at most quotient singularities, Namikawa's 
projectivity criterion \cite[Corollary 1.7]{namikawa} implies 
that $X$ is in fact projective.
\end{proof}

We next study the structure of $X$ via the linear system $|mkD|$ for 
$m\gg 0,$ where $k$ is a fixed positive integer such that $kD$ is Cartier, e.g., 
$k=|G|.$ In the proof of Proposition \ref{prop-vanishing} we noticed 
that there exists a sufficiently large integer $m$ such that
$\OO_D(mkD)$ is very ample and $\OO_X(mkD)$ is globally generated. 
For such $m$, let $\phi: X\ra \PP^N$ be the morphism defined by the 
complete linear system $|mkD|.$ Since $D$ is big as observed 
in the proof of Lemma \ref{XMoishezon}, by increasing $m$ if necessary,
we can assume that the map $\phi$ is birational onto its image \cite[Lemma 2.60]{km}. 
To simplify the notation, let $X'=\phi(X)$ and $D'=\phi(D).$

\begin{prop}
\label{prop-image}
After increasing $m$ further if necessary, the complex variety $X'$ is normal, the map $\phi$ is an isomorphism onto its image in a
neighborhood of $D$, and the $\QQ$-Cartier Weil divisor $D'$ is ample on $X'$. In particular, the $\phi$-exceptional set in $X$ does not intersect $D.$
\end{prop}

\begin{proof} That $D'$ is ample on $X'$ is clear by construction. The normality of $X'$ follows directly from 
\cite[Proposition 1.16]{kk} after increasing $m$ if necessary. To see that $\phi$ is an isomorphism onto its 
image in a neighborhood of $D$, we can now follow \cite[Chapter III, Proof of Theorem 4.2]{har-ample}. In fact, since we already know that $X'$ is normal, we have $\phi_*\OO_X = \OO_{X'}$, and so $f = \phi$ and $g = \text{Id}_{X'}$ in the Stein factorization $\phi = g \circ f$ considered in \cite{har-ample}.
\end{proof}

\section{Sweeping out the cone with hyperplane sections}

The compactification introduced in the previous section carries sufficient information 
to identify the complex structure of an ALE K\"ahler manifold. This information 
is encoded in suitable rings of sections. The ``sweeping out the cone with hyperplane sections" 
technique of Pinkham \cite{pink} exhibits the relation between these 
rings of sections geometrically. Our presentation of this technique closely follows \cite[Section 3.1]{kk}.
The proof of Theorem \ref{DiffTypeALE} is an application of this construction. 

\subsection{Rings of sections}
\label{rings-sections}
The structure of $X'$ is encoded in the $\proj$  construction \cite[Chapter II, Section 2]{hartshorne} 
applied to  the ring of sections of $D'.$ We are using the following definition of a section ring:

\begin{defn}
Let $Y$ be a projective variety. Let $E$ be a $\mathbb{Q}$-Cartier Weil divisor on $Y.$ 
The ring of sections of $E$ in $Y$ is 
$$
\displaystyle R(Y,E):=\bigoplus_{m\geq 0} H^0(Y,\OO_Y(mE)).
$$
\end{defn}
Multiplication of sections makes $R(Y,E)$ a graded ring 
where the homogeneous elements of degree $i$  are the sections of $H^0(Y,\OO_Y(iE)).$ 
The general results on graded rings that we need here are recalled in Appendix \ref{it} for convenience.

The graded rings $R(X', D')$ and $R(D',D'|_{D'})$ are finitely generated  $\CC$-algebras by Proposition
 \ref{ample-case} and Theorem \ref{summary} iii), respectively. Thus, their
 $\proj$-constructions \cite[Chapter II, Section 2]{hartshorne} are complex projective varieties.

To understand the relations between the projective varieties 
$\proj R(X',D')$ and $\proj R(D',D'|_{D'}),$ let  $0\neq S\in H^0(X',\OO_{X'}(k_0D'))$ 
be a homogeneous element 
of $R(X',D')$ vanishing exactly along $D',$ where $k_0$ is the least positive 
integer such that $H^0(X',\OO_{X'}(jk_0 D'))\neq 0$ for all $j \geq 1.$ 
By replacing the graded ring $R(X',D')$ with its $k_0$-Veronese graded subring $R(X',D')^{(k_0)}$ 
(see Definition \ref{def-veronese}),
using Proposition \ref{veronese-proj} and Remark \ref{deg1}, we will assume without 
losing generality that $S$ is a homogeneous element of degree $1.$

\begin{prop}
\label{key-properties}
The ring of sections $R(X',D')$ has the following properties:
\begin{itemize}
\item[i)] 
The ring of sections $R(X',D')$ is finitely generated and  
\begin{equation}
\label{fingen}
X'\simeq \proj R(X', D').
\end{equation}

\item[ii)] There exists an isomorphism of $\CC$-algebras 
\label{key-obs}
\begin{equation}
R(X',D')\simeq R(D',D'|_{D'})[S].
\end{equation}

\item[iii)] There exists an embedding $\iota: \proj R(D',D'|_{D'})\ra \proj R(X',D')$ 
such that the diagram
\begin{equation}
\label{proj-compatibility}
\begin{tikzcd}
&D' \arrow[hookrightarrow]{d}{j} \arrow{r}{\simeq} &\proj R(D',D'|_{D'})\arrow[hookrightarrow]{d}{\iota} \\
&X' \arrow{r}{\simeq} &\proj R(X',D')
\end{tikzcd}
\end{equation}
commutes, where $j: D'\ra X'$ denotes the inclusion map.
\end{itemize}
\end{prop}

\begin{proof}
i) By Proposition \ref{prop-image}, the $\mathbb{Q}$-Cartier Weil divisor $D'$ is ample on the normal variety $X'$.
As we already said above, item i) now follows from Proposition \ref{ample-case}.

ii) Notice that
by Proposition \ref{prop-vanishing} ii) and the projection formula, for every $m\geq 1$ 
we have an exact sequence of $\CC$-vector spaces
\begin{equation}
\label{surjection}
0 \to H^0(X',\mathcal{O}_{X'}((m-1)D')) \stackrel{\otimes S}{\ra} H^0(X',\OO_{X'}(mD'))\ra H^0(D',\OO_{D'}(mD'))\ra 0.
\end{equation}
In particular, this shows that the map of graded $\CC$-vector spaces  
$$
r: R(X',D')\ra R(D',D'|_{D'})
$$ 
given by restriction of sections from $X'$ to $D'$
is surjective. The map $\phi:X\ra X'$ induces an isomorphism
$R(D', D'|_{D'})\simeq R(D,D|_D)$ of rings of sections. Let now $s_1,\dots, s_p$ be algebra generators of  
$R(D', D'|_{D'})$ induced from $R(D,D|_D).$ We can assume that these generators are 
homogeneous, of positive degrees $d_1, \dots , d_p > 0,$ respectively. There exist
homogeneous elements $S_1, \dots ,S_p\in R(X',D')$ such that  
$r(S_i)=s_i,\, i=1,\dots, p.$ 

Since $H^0(X',\OO_{X'}(mD'))$ are finite-dimensional vector 
spaces, we can use \eqref{surjection} to prove by induction on $m$ that  $\{S,S_1,\dots, S_p\}$ are algebra 
generators of $R(X',D').$ Hence $R(X',D')\simeq R(D',D'|_{D'})[S]$ as  $\CC$-algebras, as $S$ 
is a generator in degree $1.$

iii) Since $\OO_{X'}(D')$ and 
$\OO_{D'}(D')$ are ample $\QQ$-Cartier Weil divisors on $X'$ and $D'$ 
respectively, by Proposition \ref{ample-case} there exist two natural isomorphisms 
$f:X'\ra \proj R(X',D')$ and $g: D'\ra \proj R(D',D'|_{D'}).$ 
Let $d := p \cdot \text{lcm}(d_1,\dots, d_p).$ Then it follows from the proof of Proposition 
\ref{deg1-gen} that the graded $\CC$-algebra $R(D',D'|_{D'})^{(d)}$ is generated 
by sections in $H^0(D',\OO_{D'}(dD')).$ Since $\deg S=1,$ 
this implies that 
the restriction map yields a surjection of graded $\CC$-algebras 
$$
R(X',D')^{(d)}\ra R(D',D'|_{D'})^{(d)}.
$$
By \cite[Proposition ({\bf 2.9.2})]{ega2}, such a surjection is equivalent to the 
existence of an embedding 
$$
 \proj R(D',D'|_{D'})^{(d)}\hookrightarrow\proj R(X',D')^{(d)}.
$$
Furthermore, by Proposition  \ref{veronese-proj}, this is equivalent 
to an embedding 
$$
\iota: \proj R(D',D'|_{D'})\hookrightarrow \proj R(X',D').
$$ 
By construction, the embedding $\iota$ is compatible with the embedding 
of $j: D'\hookrightarrow X'$ via the maps $f$ and $g$ defined above. 
\end{proof}

\begin{rmk} 
\label{no-prime}
Proposition \ref{prop-vanishing} ii) and the projection formula yield natural 
isomorphisms of graded $\CC$-algebras 
$$
 R(X',D')\simeq  R(X,D)\quad{\text{and}}\quad  R(D',D'|_{D'})\simeq R(D,D|_{D}).
$$
\end{rmk}
 
\subsection{Specialization to the quotient singularity}
\label{normal-cone}

In this section we identify the structure of the affine variety 
$X'\setminus D'$ as a deformation of the quotient singularity 
$\CC^n/G,$ thus proving Theorem \ref{DiffTypeALE}. 
The specialization to the quotient singularity is accomplished 
by appealing to the classical ``sweeping out the cone with 
hyperplane sections" construction of Pinkham 
\cite[Section 7]{pink}, and amounts to lifting the diagram 
(\ref{proj-compatibility}) to a one-parameter deformation.
This strategy was previously
used in \cite[Section 3]{conlon-hein-3} although, in the setting of \cite{conlon-hein-3}, 
the required surjectivity of the restriction maps came from the fact that $D'$ was an
anti-canonical divisor in $X'$.

\begin{defn} Let $Y$ be a projective variety equipped with an ample $\mathbb{Q}$-Cartier Weil divisor $E$.
\begin{itemize}
\item[i)] The affine cone over $Y$ with co-normal sheaf $\mathcal{O}_Y(E)$ is  
$$
C_a(Y,E):=\spec R(Y,E).
$$
\item[ii)]  The projective cone over $Y$ with co-normal sheaf $\mathcal{O}_Y(E)$ is
$$
C_p(Y,E):= \proj R(Y,E)[Z],
$$
where the graded ring $R(Y,E)[Z]$ is defined as 
$$
R(Y,E)[Z]:=\bigoplus_{m\geq 0} \left(\bigoplus_{r=0}^m 
H^0(Y,\mathcal{O}_Y(rE))\cdot Z^{m-r}\right).
$$
The variable $Z$ is considered of degree $1$ in the graded 
ring $R(Y,E)[Z].$
\end{itemize}
\end{defn}

For clarity, we also recall the standard definition of a deformation:

\begin{defn}
Let $Y$ be a complex variety. A deformation of $Y$ is a flat holomorphic  map
$\pi: \VV\ra T$ of complex varieties, together with an isomorphism $\pi^{-1}(0)\simeq Y,$ 
for some point $0\in T$. We say that $Y$ is deformed into the nearby fiber 
$\pi^{-1}(t),$ for $t\in T\setminus\{0\}.$
\end{defn}

\begin{proof}[Proof of Theorem \ref{DiffTypeALE}]
Consider the family $\pi:\mathcal Z\ra \CC$ of projective varieties defined by 
restricting the trivial product family $\text{pr}_{\CC}: C_p(X',D') \times \CC \ra \CC$ to the set
$$
\mathcal Z := (S- tZ = 0) \subset \proj R(X',D')[Z] \times \CC_t.
$$
If $t\neq 0,$ then we can use $Z = t^{-1}S$ to 
eliminate $Z$ and find that the fiber $\cal Z_t$ is simply
$\proj R(X',D')$, i.e., $\mathcal Z_t=X'.$
If $t = 0,$ then
 $\mathcal Z_0 = \proj R(X',D')[Z]\cap (S = 0).$  Now the key point of the whole proof is 
 that thanks to Proposition \ref{key-properties}, 
$$\proj R(X',D')[Z]\cap (S = 0) = \proj R(D',D'|_{D'})[Z].$$
Thus, the 
fiber $\cal Z_0$  is  isomorphic to the projective cone  $C_p(D',\OO_{D'}(D')).$ 

Let $\mathcal D\subset \mathcal Z$ be the divisor at infinity $(Z=0),$ and let
$\pi_{\mathcal D}:\mathcal D\ra \CC$  denote the restriction of $\pi$ to $\mathcal{D}$.
As above, we see that the fiber ${\mathcal D}_t$ of $\pi_{\mathcal D}$ 
is  $\proj R(D',D'|_{D'})\simeq D',$ for any $t\in \CC.$

Now let $\phi: \mathcal Z\setminus \mathcal D\ra \CC$ denote the restriction 
of $\pi$ to $ \mathcal Z\setminus \mathcal D.$ The map $\phi$ is flat and,  
by \cite[Proposition ({\bf 8.3.2})]{ega2}, the central fiber $\phi^{-1}(0)=\cal Z_0\setminus\DD_0$ 
is the affine cone $C_a(D',\OO_{D'}(D')).$ Since $R(D',D'|_{D'})\simeq R(D,D|_D),$   
by Theorem \ref{summary} ii) we obtain $\phi^{-1}(0)=\CC^n/G.$ Furthermore, 
for $t\neq 0$ we have 
$$
\phi^{-1}(t)=\proj R(X',D')\setminus\proj R(D',D'|_{D'})= X'\setminus D'.
$$ 
Since, by Proposition \ref{prop-image}, $M = X \setminus D$ is a resolution of singularities of 
$X'\setminus D',$ the proof of Theorem \ref{DiffTypeALE} is now complete.
\end{proof}

\begin{rmk}\label{C*rmk}
The variety $\mathcal Z$ has a natural  $\CC^*$-action given by rescaling 
the variables  $S$ and $t$ by the same factor. With respect to this action, the  
deformation $\pi:\mathcal Z\ra \CC $ is $\CC^*$-equivariant. 
Morever, since the divisor $\mathcal D\subset \mathcal Z$ is fixed by the $\CC^*$-action, 
it follows that the map $\phi: \mathcal Z\setminus \mathcal D\ra \CC$ is $\CC^*$-equivariant 
as well.
\end{rmk}

\section{Applications of deformation theory} 
\label{sec-consequences}

The specialization constructed in the proof of Theorem \ref{DiffTypeALE}
yields several consequences when applied in conjunction with standard results 
in the deformation theory of isolated quotient singularities. 
We proceed by first collecting 
several facts needed in this section.
 
\begin{defn}
\label{def-def}
Let $(S,s_0)$ be a germ of a complex analytic space. A deformation $\VV\ra \SS$ of $(S,s_0)$ 
is called {\it{versal}} if for every other deformation $\WW\ra \TT$ there exists a map $\psi: \TT\ra \SS$ 
such that $\psi^*(\VV)\simeq \WW.$ A versal deformation $\VV\ra \SS$ is called {\it{miniversal}} if the 
induced map between the Zariski tangent spaces of $\TT$ and $\SS$ is uniquely determined 
by the isomorphism class of $\WW.$
\end{defn}

The miniversal deformation of an isolated singularity $(S,s_0)$ exists and 
is unique up to (a non-unique) isomorphism, due to results of Grauert \cite{grauert}, 
Schlessinger \cite{schless-functors} and Elkik \cite{elkik}. Moreover, it is algebraic 
\cite{elkik}. Hence,  the miniversal base space $\SS$ has finitely many components, 
possibly embedded, non-reduced, and of varying dimensions.

In general, 
to pull back a given deformation from a miniversal one, one may need to shrink the 
total space $\WW$ (see \cite[Example 4.5]{artin} or \cite[page 28]{kas-schless}).
In our situation, by Remark \ref{C*rmk}, the map $\phi: \mathcal Z\setminus \mathcal D\ra \CC$ constructed in the 
preceding section is a  $\CC^*$-equivariant deformation of $\CC^n/G$. 
For isolated singularities with $\CC^*$-actions such as $\CC^n/G$, the miniversal 
deformation is $\CC^*$-equivariant \cite{pink}. That means the 
family   $\VV\ra \SS$ is $\CC^*$-equivariant. Moreover, by \cite{pink-cones}, the map 
$\psi$ in Definition \ref{def-def} can be taken 
to be $\CC^*$-equivariant with respect to the $\CC^*$-action on $\mathcal{Z}\setminus\mathcal{D}$. By construction, this action has the crucial property that 
it contracts $\mathcal{Z} \setminus \mathcal{D}$ into arbitrarily small neighborhoods of the singularity of 
the central fiber. This implies that the equivariant map $\psi$ is defined globally.
This point was also noticed and exploited by Kronheimer in a similar context in \cite[page 681]{kron2} 
and in \cite[page 692]{kron3}. 

\begin{proof}[Proof of Corollary \ref{ALEsfK}]
This follows from Theorem \ref{DiffTypeALE} and 
Schlessinger's theorem \cite{schless}, which says that the isolated quotient singularities
$\CC^n/G$ are rigid if  $n\geq 3$. 
\end{proof}

Another immediate consequence of our construction shows that in 
dimension $2$,  the singularities of $X'\setminus D'$ are rather restricted. 

\begin{cor} If $G$ is a finite subgroup of $U(2)$ acting freely on $\CC^2\setminus \{0\},$ 
then $X'\setminus D'$ has at most quotient singularities. Moreover, if $G$ is a 
finite cyclic subgroup of $U(2),$ then $X'\setminus D'$ has at most cyclic quotient 
singularities.
\end{cor}

\begin{proof} 
In the case of an arbitrary finite subgroup of $U(2),$ the result is the 
positive answer given by Esnault and Viehweg \cite{ev} to a conjecture of 
Riemenschneider. The case of a finite cyclic subgroup is Corollary 7.15 in \cite{ksb}. 
\end{proof}
 
\begin{proof}[Proof of Theorem \ref{finiteness}] 
It suffices to show that the number 
of diffeomorphism types of the minimal resolution of the fibers
of  the miniversal deformation of the quotient singularity $\CC^2/G$  
 is finite. 
Let $\psi: \VV\ra \SS$ denote the miniversal deformation of $\CC^2/G.$ 
By the algebraicity of the miniversal deformation space 
\cite{artin, elkik}, $\SS$ has finitely many irreducible components and 
the singular set of $\psi$
is proper over $\SS.$ 
For a generic point $s\in \SS,$ we 
minimally resolve the fiber $\psi^{-1}(s).$
This resolution automatically extends to a simultaneous resolution over a Zariski open subset $\SS^0 \subset \SS$ 
by the flatness of $\psi.$ All fibers over $\SS^0$ are diffeomorphic by Ehresmann's fibration theorem. 
Noetherian induction  yields the desired finiteness.
\end{proof}

Our last result is Corollary \ref{AEK}. Its proof is once again a consequence of
standard deformation theory results applied to the map $\pi:\cal Z\ra \CC.$ 

\begin{proof}[Proof of Corollary \ref{AEK}] Let $M$ be an asymptotically Euclidean 
K\"ahler manifold of complex  dimension $n.$ Since the group acting at infinity 
is trivial, the construction summarized in Theorem \ref{summary} yields a compactification 
divisor $D\simeq \CC\PP^{n-1}$ with normal bundle $\OO_D(D)\simeq\OO_{\CC\PP^{n-1}}(1).$
By Remark \ref{no-prime}, we have 
$$
R(D',\OO_{D'}(D'))\simeq R(D,\OO_{D}(D))\simeq R(\CC\PP^{n-1},\OO_{\CC\PP^{n-1}}(1)),
$$ 
and so its projective cone is $C_p  (\CC\PP^{n-1},\OO_{\CC\PP^{n-1}}(1))\simeq \CC\PP^n.$ 

Consider now the flat map $\pi:\cal Z\ra \CC$ 
constructed in the proof of Theorem \ref{DiffTypeALE}. Then the central fiber $\pi^{-1}(0),$ which is the projective cone   $C_p  (\CC\PP^{n-1},\OO_{\CC\PP^{n-1}}(1)),$ 
is isomorphic to $\CC\PP^n.$ However, as observed in Section \ref{compact-ALE},    
projective space is rigid 
under small deformations. Therefore, the fiber $\pi^{-1}(t)$ must also be isomorphic 
to $\CC\PP^n.$ Corollary \ref{AEK} now follows from Proposition \ref{prop-image}.
\end{proof}

\appendix

\label{appendix}

{

\setcounter{section}{0}
\setcounter{thm}{0}
\renewcommand{\thesection}{\Alph{section}}
\renewcommand{\thethm}{\thesection.\arabic{thm}}

\section{Graded Rings}
\label{it}
\setcounter{thm}{0}

In the context of the present study, we collect in this appendix 
some classical results on graded rings. We include their proofs 
for the convenience of the reader.

\begin{defn}
\label{def-veronese}
A graded ring $R$ is a ring together with a direct sum decomposition 
$$
R = \bigoplus_{n\geq 0} R_n,
$$
where each $R_n$ is an additive subgroup of $R$  and  
$R_mR_n\subset R_{m+n}$ for all $m,n\in \ZZ_{\geq 0}.$ If $d\in\ZZ_{> 0}$, 
the $d$-Veronese graded subring of $R$ is defined as $R^{(d)}=\bigoplus_{n\geq 0} R_{dn}.$  
\end{defn}

By definition, if $R$ as above is a graded ring, then $R_0$ is a ring, each $R_n$ is an $R_0$-module, 
and $R$ is an $R_0$-algebra. Moreover, it can be proved that $1 \in R_0$.

\begin{ex}
If $S$ is an irreducible algebraic variety and if $H$ is a $\QQ$-Cartier Weil divisor on $S$, 
we can form the section ring 
$$ 
R = R(S,H)=\bigoplus_{n\geq 0} H^0(S,\mathcal{O}_S(nH)).
$$
This is a graded ring with $R_0=H^0(S,\OO_S)\simeq \CC.$ If 
$0\neq s\in H^0(S,\mathcal{O}_S(mH))$ and $0\neq t\in H^0(S,\mathcal{O}_S(nH)),$ then 
$s\otimes t \in H^0(S,\mathcal{O}_S((m+n)H))$ is non-zero because $S$ is irreducible.
Hence the section ring $R$ is a domain.
\end{ex} 

\begin{lemma}
\label{veronese-fg}
Let $R = \bigoplus_{n\geq 0} R_n$ be a graded ring. Assume that $R_0$ is Noetherian. 
If $R$ is finitely generated as an $R_0$-algebra, 
then for all $d \in \ZZ_{>0}$, $R^{(d)}$ is finitely generated as an $R_0$-algebra as well.
Conversely, if $R$ is a domain and if $R^{(d)}$ is finitely generated as an $R_0$-algebra for some 
$d \in\ZZ_{>0}$, then $R$ is finitely generated as an $R_0$-algebra.
\end{lemma}

\begin{proof} Let $r_1, \dots, r_p$ be generators of $R$ as an $R_0$-algebra, which we may assume are nonzero and 
homogeneous with positive degrees $d_1,\dots ,d_p.$  
For $i = 1,\ldots, p$ let $\ell_i:=\lcm(d, d_i)$ and consider the elements 
$s_i := r_i^{\ell_i/d_i}\in R^{(d)}.$ Let $R'$ be the $R_0$-subalgebra of $R^{(d)}$ generated by $s_1, \ldots, s_p$. 
Then the finite set of all $r_1^{m_1}\cdots r_p^{m_p}$ with $m_i = 0, \dots, (\ell_i/d_i)-1$ obviously generates $R$ 
as an $R'$-module. In particular, $R$ is finitely generated as an $R^{(d)}$-module. It now follows from 
\cite[Proposition 7.8]{atiyah-mac} that $R^{(d)}$ is finitely generated as an $R_0$-algebra.

Conversely, suppose that $R^{(d)}$ is finitely generated as an
$R_0$-algebra. We will prove that if $R$ is a domain, then $R$ is finitely generated as an $R^{(d)}$-module, 
which obviously implies that $R$ is finitely generated as an $R_0$-algebra. To this end, observe that 
$R = \bigoplus_{m=0}^{d-1} M_m$, where each $M_m :=\bigoplus_{n\geq0}R_{m+nd}$ is an $R^{(d)}$-submodule 
of $R$. Thus, it suffices to prove that $M_m$ is finitely generated as an $R^{(d)}$-module. This is trivial if $M_m = 0.$ 
Otherwise, let $0\neq h\in M_m.$  Since $R$ is a domain, multiplication by $h^{d-1}$ 
induces an injective $R^{(d)}$-linear map $M_m\hookrightarrow R^{(d)}.$
Since $R^{(d)}$ is Noetherian by the Hilbert Basis Theorem, the desired conclusion follows.
\end{proof}

We would like $R$ to be generated by $R_1.$ This is needed in  Proposition \ref{ample-case} 
below, which is an essential ingredient in the proof of Lemma \ref{fingen} 
(see also \cite[Section 2]{ega2}).  The graded ring $R$ 
being generated by degree $1$ homogeneous elements is equivalent to 
$R$ being a graded quotient of a polynomial algebra over $R_0$ with its usual grading, 
i.e., $\proj R$ being embedded as a closed subset in some projective space over 
$R_0.$ The following standard procedure shows how to 
overcome this difficulty and ensure generation by homogeneous elements of 
degree $1$.

\begin{prop}
\label{deg1-gen}
Let $R = \bigoplus_{n\geq 0} R_n$ be a graded ring. Assume that $R$ is finitely generated as an $R_0$-algebra.
Then for all $d \in \ZZ_{>0}$ that are sufficiently divisible, $R^{(d)}$ is generated by a finite subset of $R_d$ as an $R_0$-algebra.
\end{prop}
 
\begin{proof} Let $r_1,\dots,r_p$ 
be algebra generators of $R$ over $R_0.$ We can assume that these are homogeneous 
with positive degrees $e_1,\dots ,e_p.$ Let $e$ be an arbitrary positive integer multiple of $\text{lcm}(e_1, \dots, e_p)$ 
and let $d := pe$. We will prove by induction that for all $n \in \ZZ_{> 0}$, every element $s \in R_{nd}$ can be written 
as a polynomial with $R_0$-coefficients in terms of elements of $R_d$. This is trivial for $n = 1$. Assuming the statement 
holds for some $n$, let $s \in R_{(n+1)d}$. By assumption, $s$ can be written as a polynomial with $R_0$-coefficients 
in terms of $r_1, \dots, r_p$. Each monomial has degree $(n+1)d > pe$, so in each monomial at least one of the $r_i$ 
must occur to some power $> e/e_i$. Define $s_i := r_i^{e/e_i}\in R_{e}$. It follows that each monomial can be written 
as $s_i$ (for some $i = 1, \dots, p$) times an element of $R_{(n+1)d-e}$. We apply this argument recursively $p$ times, 
which is possible because $(n+1)d - (p-1)e$ is still greater than $pe$. In this way, each monomial is seen to be of the 
form $s_1^{a_1} \cdots s_p^{a_p} t$ with $a_1 + \cdots + a_p = p$ and $t \in R_{(n+1)d - ep} = R_{nd}$.  
Because $s_1^{a_1}\cdots s_p^{a_p} \in R_{ep} = R_d$, we can now complete the inductive step by applying the 
inductive hypothesis.

Note that $R_d$ is finitely generated as an $R_0$-module. Choose an arbitrary finite set of $R_0$-module generators 
of $R_d$ and add elements of the form $s_1^{a_1}\cdots s_p^{a_p}$ with $a_1 + \cdots + a_p = p$ to this set if necessary. 
Then the above argument shows that the resulting finite subset of $R_d$ generates $R^{(d)}$ as an $R_0$-algebra.
\end{proof} 

\begin{rmk}
\label{deg1}
We can restate Proposition \ref{deg1-gen} as saying that
$R^{(d)}=\bigoplus_{n\geq 0}R_{nd}$
is generated as an $R_0$-algebra by finitely many elements of degree $1.$ 
\end{rmk} 

\begin{prop}\emph{(}\cite[Proposition ({\bf 2.4.7})]{ega2}\emph{)}
\label{veronese-proj}
Let $R=\bigoplus_{n\geq 0}R_{n}$ be a graded ring. For every integer $d,$ 
there exists a canonical isomorphism $\proj R\simeq \proj R^{(d)}.$
\end{prop}
\begin{proof}
See {\it{loc.cit}}.
\end{proof}

\begin{prop}
\label{ample-case}
Let $S$ be a normal, connected projective variety. Let $H$ be an ample $\mathbb{Q}$-Cartier Weil divisor on $S.$ 
Then the section ring $R(S,H)$ is finitely generated and there exists a canonical isomorphism 
$$
S\simeq \proj R(S,H).
$$
\end{prop}

\begin{proof} Let $d$ be a large enough positive integer such 
that the $\QQ$-line bundle $H^{\otimes d}$ is an honest line bundle, is very ample, and defines a projectively 
normal embedding of $S$. For the existence of such a $d$, see 
\cite[Chapter II, Exercise 5.14]{hartshorne}. 
Let $\iota: S\ra \PP(H^0(S,H^{\otimes d})^*)=\PP^N$  be the embedding of $S$ 
given by a basis of sections of $H^{\otimes d},$ and let $S'=\iota(S).$ Since 
$H^{\otimes d}\simeq \iota^*\OO_{\PP^N}(1)$, we have
an isomorphism of $\CC$-algebras $R(S,H^{\otimes d})\simeq R(S',\OO_{S'}(1)).$
Since $S'$ is projectively normal in $\PP^N,$ using again \cite[Chapter II, Exercise 5.14]{hartshorne}, 
we see that the section ring $R(S',\OO_{S'}(1))$ is a quotient of the 
homogeneous coordinate ring of $\PP^N$ and we have an isomorphism 
$S'\simeq \proj  R(S',\OO_{S'}(1))\simeq \proj R(S,H^{\otimes d}).$ 
In particular, $R(S',\OO_{S'}(1))$ is finitely generated and so is $R(S,H^{\otimes d}).$ 
Hence, by  Lemma \ref{veronese-fg}, we can conclude that $R(S,H)$ is finitely 
generated.  Furthermore, by Proposition \ref{veronese-proj}, we have 
$S\simeq \proj R(S,H).$
\end{proof}

}

\providecommand{\bysame}{\leavevmode\hbox
to3em{\hrulefill}\thinspace}

\end{document}